\documentclass{article}
\usepackage{hyperref}
\usepackage{
	amscd,
	amsfonts,
	amsmath,
	amssymb,
	amsthm,
	array,
	caption,
	color,
	complexity,
	enumerate,
	graphicx,
	latexsym,
	mathrsfs,
	sectsty,
	subcaption,
	tikz-cd,
	verbatim,
	wrapfig
}
\allsectionsfont{\centering}

\newcommand{\eps}{\varepsilon}

\renewcommand{\D}{\mathscr D}
\renewcommand{\A}{\mathcal A}
\newcommand{\B}{\mathcal B}
\newcommand{\restrict}{\upharpoonright}
\newcommand{\myUrl}[1]{\begin{center}
{\small\url{#1}} 
\end{center}}

\newtheorem{thm}{Theorem}

\newtheorem{lem}[thm]{Lemma}
\newtheorem{pro}[thm]{Proposition}

\theoremstyle{definition}
\newtheorem{df}[thm]{Definition}
\theoremstyle{remark}

\newtheorem{que}{Question}

\begin{document}
	\title{
		Permutations of the integers induce only the
		trivial automorphism of the Turing degrees
	}
	\author{
		Bj{\o}rn Kjos-Hanssen\\ Department of Mathematics\\ University of Hawai\textquoteleft i at M\=anoa\\ {\tt bjoern.kjos-hanssen@hawaii.edu}\footnote{
			This work was partially supported by
			a grant from the Simons Foundation (\#315188 to Bj\o rn Kjos-Hanssen).
			The author acknowledges the support of the Institut f\"ur Informatik at the University of Heidelberg, Germany
			during the workshop on \emph{Computability and Randomness}, June 15 -- July 9, 2015.
		}
	}
	\maketitle{}
	\begin{abstract}
		Let $\pi$ be an automorphism of the Turing degrees induces by
		a homeomorphism $\varphi$ of the Cantor space $2^\omega$ such that $\varphi$ preserves all Bernoulli measures.
		It is proved that $\pi$ must be trivial.
		In particular, a permutation of $\omega$ can only induce the trivial automorphism of the Turing degrees.
	\end{abstract}
	\tableofcontents
	\newpage
	\section{Introduction}
	Let $\mathscr D_{\mathrm{T}}$ denote the set of Turing degrees and let $\le$ denote its ordering.
	This article gives a partial answer to the following famous question.
	\begin{que}\label{rigid}
		Does there exist a nontrivial automorphism of $\mathscr D_{\mathrm{T}}$?
	\end{que}
	\begin{df}
		A bijection $\pi:\mathscr D_{\mathrm{T}}\to\mathscr D_{\mathrm{T}}$ is an \emph{automorphism} of $\mathscr D_{\mathrm{T}}$ if
		for all $\mathbf x, \mathbf y\in\mathscr D_{\mathrm{T}}$, $\mathbf x\le\mathbf y$ iff $\pi(\mathbf x)\le\pi(\mathbf y)$.
		If moreover there exists an $\mathbf x$ with $\pi(\mathbf x)\ne\mathbf x$ then $\pi$ is \emph{nontrivial}.
	\end{df}
	Question \ref{rigid} has a long history.
	Already in 1977, Jockusch and Solovay \cite{MR0432434} showed that each jump-preserving automorphism of the Turing degrees is the identity above $\mathbf 0^{(4)}$.
	Nerode and Shore 1980 \cite{Nerode.Shore:80} showed that
	each automorphism (not necessarily jump-preserving) is equal to the identity on some cone.
	Slaman and Woodin \cite{Slaman.Woodin:08}
	showed that each automorphism is equal to the identity on the cone above $\mathbf 0''$.

	Haught and Slaman \cite{Haught.Slaman:97} used permutations of the integers to obtain
	automorphisms of the polynomial-time Turing degrees in an ideal (below a fixed set).
	\begin{thm}[Haught and Slaman \cite{Haught.Slaman:97}]
		There is a permutation of $2^{<\omega}$, or equivalently of $\omega$, that induces a nontrivial automorphism of
		\[
			(\mathsf{PTIME}^A,\le_{\mathrm{pT}}).
		\]
		for some $A$.
	\end{thm}

			Our result can be seen as a contrast to the following work of Kent.
			\begin{df}
				$A\subset \omega$ is \emph{cohesive} if
				for each recursively enumerable set $W_e$, either $A\cap W_e$ is finite or $A\cap(\omega\setminus W_e)$ is finite.
			\end{df}

			\begin{thm}[Kent {\cite[Theorem 12.3.IX]{Rogers}, \cite{KentTAMS,KentPhD}}]\label{Kent}
				There exists a permutation $f$ such that
				\begin{enumerate}[(i)]
					\item for all recursively enumerable $B$, $f(B)$ and $f^{-1}(B)$ are recursively enumerable
					(and hence for all recursive $A$, $f(A)$ and $f^{-1}(A)$ are recursive);
					\item  $f$ is not recursive.
				\end{enumerate}
			\end{thm}
			\begin{proof}
				Kent's permutation is just any permutation of a cohesive set
				(and the identity off the cohesive set).
			\end{proof}

	\section{Universal algebra setup}
		\begin{df}\label{pullback}
			The \emph{pullback} of $f:\omega\rightarrow\omega$ is $f^*:\omega^\omega\rightarrow \omega^\omega$ given by
			\[
				f^*(A)(n) = A(f(n)).
			\]
			We often write $F=f^*$.
			Given a set $S\subseteq\omega$ let $\D_S = S^\omega/\equiv_{\mathrm{T}}$.
			Thus the elements of $\D_S$ are of the form
			\[
				[g]_S = \{\,h\in S^\omega \mid h\equiv_{\mathrm{T}} g\,\},\qquad g\in S^\omega.
			\]
			Given $F:S^\omega\to S^\omega$, let $F_S:\D_S\rightarrow\D_S$ be defined by
			\[
				F_S([A]_S) = [F(A)]_S.
			\]
			If $F=f^*_S$ then we say that $F_S$ and $F$ are both \emph{induced} by $f$.
		\end{df}

		\begin{lem}\label{vivaldi}
			For each $f:\omega\rightarrow\omega$ and each $S\subseteq\omega$,
			the pullback $f^*$ maps $S^\omega$ into $S^\omega$.
		\end{lem}
		\begin{proof}
			\[
				A\in S^\omega,\, n\in\omega \quad\Longrightarrow\quad f^*(A)(n) = A(f(n)) \in S.\qedhere
			\]
		\end{proof}
		In light of Lemma \ref{vivaldi}, we can define:
		\begin{df}
			$f^*_S:\D_S\to\D_S$ is the map given by
			\[
				f^*_S([g]_S) = [f^*(g)]_S.
			\]
		\end{df}

		For $S\subseteq\omega$ (with particular attention to $S\in\{2,\omega\}$), let
		\[
			\D_S = S^\omega / \equiv_{\mathrm{T}}.
		\]
		Our main result concerns $\D_2$; the corresponding result for $\D_\omega$ is much easier:
		\begin{thm}
			Let $f:\omega\to\omega$ be a bijection and let $f^*$ be its pullback.
			If $f^*_S$ is an automorphism of $\D_S$ for some infinite computable set $S$, then $f$ is computable.
		\end{thm}
		\begin{proof}
			Let $\eta:\omega\rightarrow S$ be a computable bijection between $\omega$ and $S$.
			Then for all $x\in\omega$,
			\[
				f^*(\eta\circ f^{-1})(x)=(\eta\circ f^{-1})(f(x)) = \eta(f^{-1}(f(x))) = \eta(x).
			\]
			Since $\eta\in S^\omega$ is computable and $f^*_S$ is an automorphism,
			$\eta\circ f^{-1}\in S^\omega$ must be computable.
			Hence $f$ is computable.
		\end{proof}

	\section{Permutations preserve randomness}
			\begin{thm}\label{bioquant}
				If $B$ is $f$-$\mu_p$-random, $F=f^*$ and $A=F(B)$ or $A=F^{-1}(B)$, then $A$ is $f$-$\mu_p$-random.
			\end{thm}
			\begin{proof}
				First note that
				$f^{-1}$-$\mu_p$-randomness is the same as $f$-$\mu_p$-randomness since $f\equiv_{\mathrm{T}} f^{-1}$.
				Thus the result for $A=F^{-1}(B)$ follows from the result for $A=F(B)$.
				So suppose $A=F(B)$ and $A$ is not $f$-$\mu_p$-random.
				So $A\in\cap_n U_n$ where $\{U_n\}_n$ is an $f$-$\mu_p$-ML test. Then
				\[
					B\in \{X \mid F(X)\in\cap_n U_n\} = \cap_n V_n
				\]
				where
				\[
					V_n = \{X \mid F(X)\in U_n\} = F^{-1}(U_n)
				\]
				We claim that $V_n$ is $\Sigma^0_1(f)$ (uniformly in $n$) and $\mu_p(V_n)=\mu_p(U_n)$.
				Write $U_n=\cup_k [\sigma_k]$ where the strings $\sigma_k$ are all incomparable.
				Then
				\[
					V_n = \cup_k F^{-1}([\sigma_k])
				\]
				and
				\[
					\mu_p [\sigma_k] = \mu_p F^{-1}([\sigma_k])
				\]
				and the $F^{-1}([\sigma_k])$, $k\in\omega$ are still disjoint and clopen.
				(If we think of $\sigma\in 2^{<\omega}$ as a partial function from $\omega$ to $2$ then
				\[
					F^{-1}([\sigma]) = \{X \mid F(X) \in [\sigma]\}
				\]
				\[
					= \{X\mid X(f(n)) = \sigma(n), n<|\sigma|\}
					= [\{\langle f(n), \sigma(n)\rangle \mid n<|\sigma|\}].)
				\]
				Thus $\{V_n\}_n$ is another $f$-$\mu_p$-ML test, and so
				$B$ is not $f$-$\mu_p$-random, which completes the proof.
			\end{proof}

		\begin{thm}\label{Claim1}
			$\mu_p(\{A: A\ge_{\mathrm{T}} p\})=1$, in fact if $A$ is $\mu_p$-ML-random then $A$ computes $p$.
		\end{thm}
		\begin{proof}
			Kjos-Hanssen \cite{K:2010} showed that each Hippocratic $\mu_p$-random set computes $p$.
			In particular, each $\mu_p$-random set computes $p$.
		\end{proof}
	\section{Cones have small measure}
		\begin{df}[Bernoulli measures]
			For each $n\in\omega$, 
			$$\mu_p(\{X\in 2^\omega: X(n)=1\})=p$$
			and $X(0),X(1),X(2),\ldots$ are mutually independent random variables.
		\end{df}

		\begin{df}
			An \emph{ultrametric} space is a metric space with metric $d$ satisfying the strong triangle inequality
			\[
				d(x, y)\le\max\{d(x, z), d(z, y)\}.
			\]
		\end{df}
		\begin{df}
			A \emph{Polish space} is a separable completely metrizable topological space.
		\end{df}
		\begin{df}
			In a metric space, $B(x,\eps)=\{y: d(x,y) < \eps\}$.
		\end{df}
		\begin{thm}[\texorpdfstring{\cite[Proposition 2.10]{BenMiller}}{}]\label{LDT}
			Suppose that $X$ is a Polish ultrametric space,
			$\mu$ is a probability measure on $X$, and
			$\A\subseteq X$ is Borel. Then
			\[
				\lim_{\eps\to 0}\frac{\mu(\A\cap B(x,\eps))}{\mu(B(x,\eps))}=1
			\]
			for $\mu$-almost every $x\in \A$.
		\end{thm}
		\begin{df}
			For any measure $\mu$ define the conditional measure by
			\[
				\mu(\A\mid\B) = \frac{\mu(\A\cap\B)}{\mu(\B)}.
			\]
			A measurable set $\A$ has density $d$ at $X$ if
			\[
				\lim_n \mu_p(\A\mid [X\restrict n]) = d.
			\]
		\end{df}
		Let $\Xi(\A) = \{X: \A\text{ has density }1\text{ at }X\}$.
		\begin{thm}[Lebesgue Density Theorem for $\mu_p$]\label{cold-brewed}
			For Cantor space with Bernoulli($p$) product measure $\mu_p$, the Lebesgue Density Theorem holds:
			\[
				\lim_{n\to\infty}\frac{\mu_p(\A\cap [x\restrict n])}{\mu_p([x\restrict n])} = 1
			\]
			for $\mu$-almost every $x\in \A$.

			If $\A$ is measurable then so is $\Xi(\A)$.
			Furthermore, the measure of the symmetric difference of $\A$ and $\Xi(\A)$ is zero, so
			$\mu(\Xi(\A))=\mu(\A)$.
		\end{thm}
		\begin{proof}
			Consider the ultrametric $d(x,y)=2^{-\min\{n:x(n)\ne y(n)\}}$.
			It induces the standard topology on $2^\omega$. Apply Theorem \ref{LDT}.
		\end{proof}
		Sacks \cite{Sacks:63} and de Leeuw, Moore, Shannon, and Shapiro \cite{deLeeuw} showed that
		each cone in the Turing degrees has measure zero. Here we use Theorem \ref{cold-brewed} to extend this to $\mu_p$.
		\begin{thm}\label{ce}
			If $\mu_p(\{X: W_e^X=A\})>0$ then $A$ is c.e. in $p$.
		\end{thm}
		\begin{proof}
			Suppose $\mu_p(\{X: W_e^X=A\})>0$.
			Then $S := \{X \mid W_e^X = A\}$ has positive measure, so $\Xi(S)$ has positive measure,
			and hence by Theorem \ref{LDT} there is an $X$ such that $S$ has density 1 at $X$.
			Thus, there is an $n$ such that
			$\mu_p(S\mid [X\restrict n]) > \frac12$.
			Let $\sigma = X\restrict n$.
			We can now enumerate $A$ using $p$ by taking a ``vote'' among the sets extending $\sigma$.
			More precisely, $n\in A$ iff
			\[
				\mu_p(\{Y: \sigma\prec Y\wedge n\in W_e^Y\}) > \frac12,
			\]
			and the set of $n$ for which this holds is clearly c.e. in $p$.
		\end{proof}
		\begin{thm}\label{Claim3}
			Each cone strictly above $p$ has $\mu_p$-measure zero:
			\[
				\mu_p(\{A: A\ge_{\mathrm{T}} q\})=1\qquad\Longrightarrow\qquad q\le_{\mathrm{T}} p.
			\]
		\end{thm}
		\begin{proof}
			If $A$ can compute $q$ then $A$ can enumerate both $q$ and the complement of $q$.
			Hence by Theorem \ref{ce}, $q$ is both c.e. in $p$ and co-c.e. in $p$; hence $q\le_{\mathrm{T}} p$.
		\end{proof}
	\section{Main result}
		We are now ready to prove our main result Theorem \ref{main} that
		no nontrivial automorphism of the Turing degrees is induced by a permutation of $\omega$.
		\begin{thm}\label{main}
			If $\pi$ is an automorphism of $\D_2$ which is induced by a permutation of $\omega$ then
			$\pi(\mathbf p)=\mathbf p$ for each $\mathbf p\in\D_{\mathrm{T}}$.
		\end{thm}
		\begin{proof}
			Fix a permutation $f:\omega\to\omega$ and let $F=f^*\restrict 2^\omega$.
			Let $B$ be $f$-$\mu_p$-random.
			We claim that $B$ computes $F(p)$.

			By Theorem \ref{Claim1}, for any $f$-$\mu_p$ random $A$, we have
			$p\le_{\mathrm{T}} A$, hence $F(p)\le_{\mathrm{T}} F(A)$. So it suffices to represent $B$ as $F(A)$.

			Now $B = F(F^{-1}(B))$. Let $A = F^{-1}(B)$. By Theorem \ref{bioquant}, $A$ is $f$-$\mu_p$-random.
			Thus every $f$-$\mu_p$-random computes $F(p)$.

			Thus we have completed the proof of our claim that $\mu_p$-almost every real computes $F(p)$.

			By
			Theorem \ref{Claim3} it follows that $F(p)\le_{\mathrm{T}} p$.
			
			By considering the inverse $f^{-1}$ we also obtain $F^{-1}(p)\le_{\mathrm{T}} p$ and hence $p\le_{\mathrm{T}} F(p)$.
			So $F(p)\equiv_{\mathrm{T}} p$ and $F$ induces the identity automorphism.
		\end{proof}

	\section{Computing the permutation}
		\begin{thm}\label{GuestHouse4426}
			Let $f:\omega\to\omega$ be a permutation.
			Let $F=f^*$ be its pullback (Definition \ref{pullback}) to $2^\omega$.
			If for positive Lebesgue measure many $G$, $F(G)\le_{T} G$,
			then $f$ is recursive.
		\end{thm}
		\begin{proof}
			By the Lebesgue Density Theorem we can get a $\Phi$ and a $\sigma$ such that,
			if $\mu_\sigma$ denotes conditional probability on $\sigma$
			and $E = \{A: F(A)=\Phi^A\}$, then
			\[
				\mu_\sigma(E)\ge 95\%.
			\]
			For simplicity let us write $p_n(A) = A+n=A\cup\{n\}$ and $m_n(A) = A-n=A\setminus\{n\}$.
			Then $p_n^{-1}E = \{A: p_n(A)\in E\}$. Note that
			\[
				E \subseteq p_n^{-1}(E)\cup m_n^{-1}(E)
			\]
			and
			\[
				E^c \subseteq p_n^{-1}(E^c)\cup m_n^{-1}(E^c)
			\]
			Then
			\[
				\mu_{\sigma}(E) \le \mu_{\sigma}(p_n^{-1}(E)\cup m_n^{-1}(E))
				\le \mu_{\sigma}(p_n^{-1}(E)) + \mu_{\sigma}(m_n^{-1}(E))
			\]
			We now have
			\[
				\mu_\sigma\{A:F(A+n)=\Phi^{A+n}\}\ge 90\%
			\]
			and
			\[
				\mu_\sigma\{A:F(A-n)=\Phi^{A-n}\}\ge 90\%;
			\]
			Indeed,
			the events $m_n^{-1}(A)$, $p_n^{-1}(A)$ are each independent of the event $n\in A$,
			so for $n>|\sigma|$, 
			\begin{eqnarray*}
				95\% \le \mu_{\sigma}(E) &=& \mu_{\sigma}(p_n^{-1}(E)\mid n\in A)\mu_{\sigma}(n\in A)
				+ \mu_{\sigma}(p_n^{-1}(E)\mid n\notin A)\mu_{\sigma}(n\notin A)\\
				&=& \frac12\left(
					  \mu_{\sigma}(p_n^{-1}(E)\mid n\in A)
					+ \mu_{\sigma}(m_n^{-1}(E)\mid n\notin A)
				\right)\\
				&=& \frac12\left(\mu_{\sigma}(p_n^{-1}(E))+\mu_{\sigma}(m_n^{-1}(E))\right)
			\end{eqnarray*}
			which gives
			\[
				1.9 \le \mu_{\sigma}(p_n^{-1}(E)) + \mu_{\sigma}(m_n^{-1}(E))
				\le 1 + \min\{\mu_{\sigma}(p_n^{-1}(E)),\mu_{\sigma}(m_n^{-1}(E))\}.
			\]

			Also $F(A-n)$ and $F(A+n)$ differ in exactly one bit, namely $f^{-1}(n)$, for all $A$:
			\begin{eqnarray*}
				F(A-n)(b)\ne F(A+n)(b)&\Longleftrightarrow& (A-n)(f(b))\ne (A+n)(f(b))\\
									  &\Longleftrightarrow& n=f(b)\Longleftrightarrow b=f^{-1}(n),
			\end{eqnarray*}
			that is
			\[
				\{A: (\forall b)(F(A+n)(b)\ne F(A-n)(b)\leftrightarrow b=f^{-1}(n))\} = 2^\omega.
			\]
			Let $D_{n,b} = \{A: \Phi^{A+n}(b)\downarrow\ne\Phi^{A-n}(b)\downarrow\}$.
			For $n>|\sigma|$,
			\[
				  \mu_\sigma \left(D_{n,f^{-1}(n)}\setminus\bigcup_{b\ne f^{-1}(n)} D_{n,b}\right)
				= \mu_\sigma\{A: (\forall b)(A\in D_{n,b}\leftrightarrow b=f^{-1}(n))\}\ge 80\%
			\]
			since
			\[
				\mu_\sigma\{A: \neg(\forall b)(A\in D_{n,b}\quad\leftrightarrow\quad b=f^{-1}(n))\}
			\]
			\[
				\le
				\mu_\sigma(\neg p_n^{-1}(E)) + \mu_\sigma(\neg m_n^{-1}(E)) \le 10\%+10\%=20\%.
			\]

			Therefore, given any $n$, we can compute $f^{-1}(n)$:
			enumerate computations until we have found some bit $b$ such that
			\[
				\mu_\sigma D_{n,b}\ge 80\%.
			\]
			Then $b=f^{-1}(n)$.

			Thus $f^{-1}$ is computable and hence so is $f$.
		\end{proof}

		\begin{thm}\label{abstract}
			If $\pi$ is an automorphism of $\D_{\mathrm{T}}$ which is induced by
			a permutation $f$ of $\omega$ then $f$ is recursive.
		\end{thm}
		\begin{proof}
			By Theorem \ref{main}, $f^*(G)\equiv_{\mathrm{T}} G$ for each $G\in 2^\omega$.
			By Theorem \ref{GuestHouse4426}, $f$ is recursive.
		\end{proof}

	\section{Measure-preserving homeomorphisms of the Cantor set}
		\begin{pro}
			A permutation of $\omega$ induces a homeomorphism of $2^\omega$ that is $\mu_p$-preserving for each $p$.
		\end{pro}
		\begin{pro}
			There exist homeomorphisms of $2^\omega$ that are $\mu_p$-preserving for each $p$, but are not induced by a permutation.
		\end{pro}
		\begin{proof}
			Map
			\[
				[1]\mapsto [111]\cup [001]\cup [101]\cup [110]
			\]
			(more generally, any collection of cylinders of strings of length 3 including 2 strings of Hamming weight 2 and 1 of Hamming weight 1).
	
			Another way to express this is that the homeomorphism preserves the fraction of 1s in a certain sense.

			More precisely,
			\begin{eqnarray*}
				100 \mapsto 001,\\
				101 \mapsto 101,\\
				110 \mapsto 110,\\
				111 \mapsto 111.
			\end{eqnarray*}
		\end{proof}

		\begin{thm}
			Suppose $\varphi$ is a homeomorphism of $2^\omega$ which is $\mu_p$-preserving for all $p$
			(it suffices to require this for infinitely many $p$, or for a single transcendental $p$).
			Suppose $\varphi$ induces an automorphism $\pi$ of the Turing degrees. Then $\pi = \mathrm{id}$.
		\end{thm}
		We omit the proof which follows along the same lines as before.
	\newpage
		\bibliography{Hippocratic-rigidity-arXiv}
		\bibliographystyle{plain}
\end{document}